\documentclass[final,3p,mathptmx]{article}

\usepackage{graphicx}
\usepackage{comment}
\usepackage{amsthm}
\usepackage{amsmath}
\usepackage{amsfonts}
\usepackage{color}

\usepackage{longtable}
\usepackage{tikz}
\usepackage{mathrsfs}

\newcommand{\N}{{\mathbb N}}

\newcommand{\R}{{\mathbb R}}
\newcommand{\Rn}{{\mathbb R}^n}

\newcommand{\Kn}{{\cal K}_n}

\newcommand{\CH}{{\mathrm{CH}}}

\newcommand{\Kone}{{\cal K}_1}
\newcommand{\Ktwo}{{\cal K}_2}

\newcommand{\dist}{{\rm dist}}

\newcommand{\haus}{{\mbox haus}}

\usepackage{amssymb}

\def\dist{\operatorname{dist}}

\newcommand{\Pair}[2]{\Pi({#1},{#2})}
\newcommand{\PrjXonY}[2]{\Pi_{#2}{(#1)} }

\newcommand{\Map}[1] {{F:#1 \rightarrow \Kn}}
\newcommand{\MapR}[1] {{F:#1 \rightarrow \Kone}}

\newtheorem{remark}{Remark}[section]
\newtheorem{definition}[remark]{Definition}

\newtheorem{example}[remark]{Example}

\newtheorem{theo}[remark]{Theorem}
\newtheorem{cor}[remark]{Corollary}

\begin{document}

\title{ {Metrically differentiable set-valued functions\\ and their local linear approximants}
}
\renewcommand{\thefootnote}{\fnsymbol{footnote}}
\author{
	Nira Dyn\footnotemark[1]\ , \;
	Elza Farkhi\footnotemark[1]\ , \;
	Alona Mokhov\footnotemark[2]\ .
}
\footnotetext[1]{Tel-Aviv University, Israel}
\footnotetext[2]{Afeka Tel-Aviv Academic College of Engineering, Israel}

\date{}
\maketitle

\centerline{ \bf Dedicated to Professor Dany Leviatan on the occasion of his 80th birthday}
\bigskip

\begin{abstract}
A new notion of metric differentiability of set-valued functions at a point is introduced in terms of right and left limits of special set-valued metric divided differences of first order.
A local metric linear approximant of a metrically differentiable set-valued function at a point is defined and studied. This local approximant may be regarded as a special realization
of the set-valued Euler approximants of M.~S.~Nikolskii and the directives of Z.~Artstein. 
Error estimates for the local metric linear approximant are obtained. 
In particular, second order approximation is derived for a class of ``strongly'' metrically differentiable set-valued maps.
\end{abstract}
\bigskip

{\bf Key words:} compact sets in Euclidean spaces, set-valued functions, metric divided differences of first order, metric differentiability, local metric linear approximants

\section{Introduction}\label{sect_Intro}

A differential calculus of set-valued functions (SVFs, set-valued maps, multifunctions) is playing a significant role in different areas of applied mathematics. 
Various notions of set-valued derivatives, motivated by different needs and applications, have been introduced and studied in the last five decades. 
The earlier papers focus on differentiabilty of multifunctions with convex images (see e.g.~\cite{Baier-Farkhi:2000, Bank-Jacobs, Blasi}).
Different concepts of set-valued differentiability (such as Hukuhara-derivatives, contingent derivatives, graphical derivatives, co-derivatives, generalized differentiability) 
are introduced, discussed, and studied, for instance, in~\cite{A-F:90, Hukuhara:67, Mordukhovich:06, Pang:11, Rockafellar:98}.
In~\cite{Gautier:90, Lemar_Zowe:91, Silin:97} the notions of affine, semi-affine, eclipsing, quasi-affine mappings are suggested for the role of set-valued "linear" approximants.
Derivatives of set-valued functions also arise in probability and statistics, where another approach still might be required (see e.g.~\cite{Khmaladze_Weil:14}). 

In a series of works~(e.g.~\cite{BDFM:19,DFM:Chains,DFM:Book_SV-Approx,DFM:MetricIntegral,DFM:HighOrder}) we have developed a metric approach to the approximation of SVFs, mapping a compact interval to general 
(not necessarily convex) compact sets in~$\Rn$. 
This approach, inspired by the ideas from~\cite{Artstein:MA}, includes metric tools such as metric chains, metric linear combinations, metric integral.
The metric tools are applied to adapt classical sample-based approximation operators to multifunctions. 
In~\cite{BDFM:19,DFM:Chains,DFM:Book_SV-Approx} we obtain approximation order at most~$1$ for SVFs which are continuous or of bounded variation.  
In~\cite{DFM:HighOrder} we define the notion of metric divided differences of any order $r\in\N$ and obtain high order approximation by metric piecewise-polynomial interpolation, 
without introducing a notion of derivative for set-valued functions.

In this paper, we introduce the concept of metric differentiability for a univariate multifunction $F$ with the aim of achieving a local approximation of $F$ with an error of order higher than 1. 
We apply this concept in constructing and analyzing a local metric linear approximant. Our approach yields error estimates beyond first order (including second order) for these approximants, 
extending classical error estimates for real-valued functions to set-valued functions.

The notion of local metric linear approximant,  introduced here may be regarded as a special unique  ``realization'' 
of the Euler approximants of M.S. Nikolskii \cite{NIK:93}, as well as of the ``directives'' of Artstein~\cite{ART:95}. 
They both have proposed a definition of a set-valued derivative by a local ``linear approximant'' of the given multifunction
in a neghborhood of a given point with error of order  ${o}(h)$ measured in the Hausdorff metric. 
They both have noticed that such a definition does not provide uniqueness of the derivative. 

We define here the notion of metric divided differences of first order anchored at a point $y\in F(x_0)$ 
and introduce one-sided (left and right) metric derivatives anchored at $y$ as the one-sided limits (left and right) of the corresponding metric divided differences.
We construct the local metric linear approximant of a metrically differentiable multifunction at a point $x_0$ and obtain rate of approximation $o(h)$. To get a higher order, $O(h^{m})$, $m>1$,
we introduce a higher order smoothness property for SVFs related to the rate of convergence of the metric divided differences to the one-sided metric derivatives.
\medskip

The paper is organized as follows:
In the next section we present preliminary material.
In Section~3 we present metric divided differences of first order, define the one-sided metric derivatives and investigate their properties. 
In Section~4 we introduce and study the local metric linear approximant.

\section{Preliminaries}
We use the folowing notation. The Euclidean norm in $\Rn$ is denoted by  $|\cdot|$.
In this paper we consider sets in $\Kn$, where $\Kn$ denotes the collection of nonempty compact subsets of $\Rn$.
The distance from a point ${x \in \Rn}$ to a set $A\subset\Rn$ is denoted by ${\mbox{dist}(x,A)=\min_{a \in A}|x-a|}$.

The set of all projections of $a \in \Rn$ on a set $B \in \Kn$ is
$$
\PrjXonY{a}{B}=\{b \in B:|a-b|=\dist(a,B)\}. 
$$
We denote the collection of all metric pairs of two sets $A,B \in \Kn$ by
$$
\Pair{A}{B} = \{(a,b) \in A \times B:\; a \in \PrjXonY{b}{A}\;\, \mbox{or}\;\, b\in\PrjXonY{a}{B} \}.
$$
\begin{remark}	\label{rem-1}
	It follows from the compactness of $A,B\in \Kn$ that $\Pair{A}{B}$ is a non-empty compact set. 
	(see Lemma 2.4 in~\cite{DFM:HighOrder})
\end{remark}
Using the metric pairs, the Hausdorff distance  can written as
\begin{equation} \label{haus_MetrPair}
	\haus(A,B)= \max \{|a-b|:\; (a,b)\in \Pair{A}{B}\}.
\end{equation}

The space $\Kn$ is a complete metric space with respect to this metric \cite{S:93}.
In this paper we regard convergence of sets and continuity of set-valued maps 
with respect to the Hausdorff metric.

We use in the paper the following property of the Hausdorff metric, which is easy to check:
\begin{equation}\label{haus_pr_uni}
	\haus (\bigcup\limits_{\lambda\in\Lambda} A_\lambda, \bigcup\limits_{\lambda\in\Lambda} B_\lambda) \le
	\sup_{\lambda\in\Lambda} \haus (A_\lambda, B_\lambda),\quad \Lambda\; \hbox{is an index set}.
\end{equation}
For a given set $A$ we use the notation
$$
\lambda A  = \left\{ \lambda a \ : \ a\in A \right\}, \quad \lambda \in \R,
$$
and
$$
\| A \| = \max \left\{ |a| \ : \ a\in A \right\}. 
$$
A linear Minkowski combination of two sets $A, B$ is
$$
\lambda A + \mu B = \{ \lambda a + \mu b: a\in A, b\in B \},\quad \lambda,\mu\in \R.
$$
The following  relations between operations on sets and the Hausdorff metric are easy to verify.
\begin{equation}\label{prop:haus}
	\haus(\{x\}+A,\{x\}+B)=\haus(A,B), \qquad \haus(\lambda A,\lambda B)=|\lambda|\haus(A,B),
\end{equation}
where $\lambda\in \R$ and $\{x\}$ is the set consisting of the single element $x$.
\medskip

\noindent We recall here the notion of a metric chain and of a metric linear combination~\cite{DFM:MetricIntegral}.

Given a finite sequence of sets $A_0, \ldots, A_m \in \Kn$, $m \ge 1$, 
the collection of all metric chains of ${A_0, \ldots, A_m}$ is 
\begin{equation}\label{def_MetrChain}
		\CH(A_0,\ldots,A_m) := \left\{ (a_0,\ldots,a_m): \; (a_i,a_{i+1}) \in \Pair {A_i}{A_{i+1}}, \  i=0,1,\ldots,m-1 \right\}.
\end{equation}
	For $m=0$ we define $\CH(A_0)=A_0$. Note that $\CH(A_0, A_1)=\Pair{A_0}{A_1}$.
\smallskip

	The metric linear combination of the sets $A_0, \ldots, A_m \in \Kn$, is 
\begin{equation}\label{def_MLC}
	\bigoplus_{i=0}^m \lambda_i A_i :=
	\left\{ \sum_{i=0}^m \lambda_i a_i \,: (a_0,\ldots,a_m) \in \CH(A_0,\ldots,A_m) \right\}.
\end{equation}

\begin{remark}\label{rem_2}
	${}$
	\begin{enumerate}
	\item [(1)] In the special case $m=1$ and $\lambda_0=1$, $\lambda_1=-1$ we obtain the {metric difference} of two 
	sets $A,B$ defined in \cite{DFM:Chains},
	$$ A\ominus B:= \{\ a-b:(a,b)\in \Pair{A}{B} \ \}.$$ 
	\item [(2)] In the special case $m=1$ and $\lambda_0, \lambda_1 \in [0,1]$, $\lambda_0 + \lambda_1 = 1$, 
		the metric linear combination is the binary operation between compact sets, introduced in~\cite{Artstein:MA}. 
	\item [(3)] The set $\displaystyle \bigoplus_{i=0}^m \lambda_i A_i$ is compact (see~\cite{DFM:HighOrder})).
	\end{enumerate}
\end{remark}

In this work we consider univariate set-valued functions with values in $\Kn$. 
The graph of a multifunction $\Map{[a,b]}$ is the set
$Graph (F)=\{(x,y) : x\in [a,b],\, y\in F(x)   \}$.
\medskip

\section{First metric divided differences and one-sided metric derivative}\label{Sect_MetricDD_MetricDeriv}

We start with defining divided differences.
\begin{definition}\label{def:MDD_Order1_anchored}
Let $F:(a,b)\to\Kn$. Fix $x_0,x\in (a,b)$,  $x_0\neq x$ and $y_0\in F(x_0)$. Define
the {\bf  
\textsl{first metric divided difference of $F$ at $x_0,x$, anchored at $y_0$}} , as the set
of vectors  
\[
[x_0,x]^M F|_{y_0}:= \  \left \{ \frac{y-y_0}{x-x_0} :\, (y_0,y)\in \Pair{F(x_0)}{F(x)} \right \}.
\]
\end{definition}

\begin{remark}\label{Remark_onDD}	${}$
	\begin{enumerate}
	\item [(i)] By the compactness of $F(x),F(x_0)$, and by Remark \ref{rem-1}, $[x_0,x]^M F|_{y_0}$ is non-empty and compact.
	\item [(ii)] If $(x_0, y_0)$ is in the interior of $Graph(F)$, then $y_0\in F(x)$ for $x$ in a small enough neighborhood of $x_0$,  thus $(y_0,y_0)\in \Pair{F(x_0)}{F(x)}$ is the only metric pair containing $y_0$, and therefore $[x_0,x]^M F|_{y_0}=\{0 \}$.
	\item [(iii)] By Definition~\ref{def:MDD_Order1_anchored}, $F(x)= \bigcup_{y_0\in F(x_0)} \{ \{y_0\}+(x-x_0)[x_0,x]^M F|_{y_0}  \}$.
\end{enumerate}
\end{remark}

\begin{example}\label{ex_ddif1}
	Let $F:(0,2)\to \Kone$ be defined by $F(x)=\{ x^\alpha,x^\beta \}$, where $\alpha, \beta \in \N$,
	$\alpha\neq \beta$, and let $x_0\in (0,2)$, $x=x_0+h$ with $|h|\neq 0$ sufficiently small. Then for $x_0, x_0+h \le 1$ or\ $x_0, x_0+h \ge 1$
	\[
	\Pair{F(x_0)}{F(x_0+h)} = \left \{ \big( x_0^\alpha , (x_0+h)^\alpha \big) , \big( x_0^\beta,(x_0+h)^\beta \big)  \right\}.
	\]
For $x_0=1$ we get $[1,1+h]^M F|_{1}=\left \{ \frac{(1+h)^\alpha-1}{h}, \frac{(1+h)^\beta-1}{h} \right \}$, while for $x_0\neq 1$ we have 
	$$
	[x_0,x_0+h]^M F|_{x_0^\alpha} = \left \{\frac{(x_0+h)^\alpha-x_0^\alpha}{h} \right\}\quad , \quad
	[x_0,x_0+h]^M F|_{x_0^\beta} = \left \{\frac{(x_0+h)^\beta-x_0^\beta}{h} \right\}.
	$$
\end{example} 

\begin{definition}\label{def:MDD_Order1}
	Let $F:(a,b)\to\Kn$. Fix $x_0,x\in (a,b)$,  $x_0\neq x$. Define
	the {\bf  \textsl{first metric divided difference of $F$ at $x_0,x$ }} as
	$$
	[x_0,x]^M F:= \bigcup\limits_{y\in F(x_0)} [x_0,x]^M F|_{y},
	$$
\end{definition}

\begin{remark}\label{rem:metr_diff}
	${}$
\begin{enumerate}
	\item [(i)] The set $[x_0,x]^M F$ is the metric divided difference of order $1$ in~\cite[Remark~5.2]{DFM:HighOrder}. It is compact as is shown in Lemma~5.5 there.
	\item [(ii)] Using the metric difference of two sets (see Remark~\ref{rem_2}), the first metric divided difference can be expressed, similarly to
	the divided difference of real-valued functions, as
	$$
		[x_0,x_1]^M F=
		\frac{1}{x_1-x_0}\left( F(x_1)\ominus F(x_0) \right).
		$$
	It is clear that one can interchange the points $x_0,x_1$ in this divided difference.
	\item [(iii)] 
	In view of \eqref{haus_MetrPair}, $\haus(A,B)=\|A\ominus B\|$ and by the second property in~\eqref{prop:haus},
a~multifunction $F$ is Lipschitz  with a constant $L$ in a domain if and only if 
$$\|[x_0,x_1]^M F\|=\frac{\haus(F(x_1),F(x_0))}{|x_1-x_0|}\le L$$ 
for any $x_0,x_1$ in this domain.
	\item [(iv)]
		For a single-valued function, Definition~\ref{def:MDD_Order1} coincides with the classical definition of divided difference of first order.
\end{enumerate}
\end{remark}

Now we introduce the right and the left metric derivatives. 
We use the notations $\lim_{x\to x_0^{+}}$ and $\lim_{x\to x_0^{-}}$ for the right and left limits, respectively.

\begin{definition} \label{def_metr_deriv}
Let $F:(a,b)\to \Kn,\ x_0\in (a,b)$. 
\begin{enumerate}
\item [(i)] We say that $F$ is {\bf \textsl {metrically differentiable from the right}}
{\bf\textsl{ at $x_0$}} if for any $y\in F(x_0)$ there is a~non-empty set $D^M_+ F(x_0)|_{y}$ of vectors  satisfying  
$$
	D^M_+ F(x_0)|_{y}=\lim_{x\to x_0^{+}}[x_0,x]^M F|_{y} 
$$
uniformly in $y\in F(x_0)$, namely, for any $\epsilon >0$  
there exists \ $\delta >0$ such that for any $x$ with $0<x-x_0<\delta$
$$
\sup_{y\in F(x_0)}\haus \left(D^M_+ F(x_0)|_{y},[x_0,x]^M F|_y\right) < \epsilon.
$$
We call the set $D^M_+ F(x_0)|_{y}$ {\bf \textsl{the right metric derivative of $F$} at $x_0$ 
\textsl{anchored at} $y\in F(x_0)$}.\\ The set
$$
{\mathfrak D}^M_+ F=  \left \{ D^M_+ F(x)|_{y}:\;\,  (x,y) \in Graph(F) \right \}
$$
determines a set-valued vector field of the right metric derivative of $F$ on the graph of $F$.\\
We call the set 
$$
D^M_+ F(x_0) = \bigcup_{y \in F(x_0)} D^M_+ F(x_0)|_{y}=\lim_{x\to x_0^{+}}\frac{1}{x-x_0}\left( F(x)\ominus F(x_0) \right)
$$
{\bf \textsl{the right metric derivative of $F$}} at $x_0$.

\item [(ii)] Similarly, we say that $F$ is {\bf \textsl {metrically differentiable from the left at $x_0$}} and 
define the sets $D^M_- F(x_0)|_{y}$\ , ${\mathfrak D}^M_- F$\ , $D^M_- F(x_0)$.

\item [(iii)] If $F$ is both metrically differentiable from the right and from the left at $x_0$, then we say that $F$ is {\bf \textsl {metrically differentiable }}at $x_0$.
\end{enumerate}
\end{definition}

Below are some observations concerning Definition~\ref{def_metr_deriv}.
\begin{remark}\label{rem_3}
${}$
\begin{itemize}
\item [(i)]  Note that in $D^M_{+} F(x_0)$, $D^M_{-} F(x_0)$ the information on the attachments to points in $Graph(F)$ is lost in contrary to ${\mathfrak D}^M_+ F$ and ${\mathfrak D}^M_- F$.
\item [(ii)] Although for single-valued functions the existence of derivative is defined when the right derivative equals the left one, for set-valued functions the equality $D^M_- F(x_0)=D^M_+ F(x_0)$ 
		is not satisfied in many cases, and we do not require it for metric differentiability. 
\item [(iii)] Note that the sets $D^M_+ F(x_0)|_{y} , D^M_- F(x_0)|_{y} ,  D^M_+ F(x_0) , D^M_- F(x_0)$ are compact as Hausdorff
	limits of compact sets (see Remark~\ref{Remark_onDD}~(i), Remark~\ref{rem_2}~(ii)) and are unique as limits in a complete metric space.
\item [(iv)]
	An equivalent definition of the right metric differentiability of $F$ at $x_0$ is the existence of
	a continuous function $\omega_{x_0}:[0,\infty) \to [0,\infty)$ with $\omega_{x_0}(0)=0$ 
	and such that for $x$ in a right neighborhood of $x_0$ 
	$$
	\sup_{y\in F(x_0)}\haus([x_0,x]^M F|_y,D^M_{+} F(x_0)|_y) \le \omega_{x_0}(|x-x_0|).
	$$
	Similarly, for the left metric differentiability of $F$ at $x_0,$ the following condition should hold
	$$
	\sup_{y\in F(x_0)}\haus([x_0,x]^M F|_y,D^M_{-} F(x_0)|_y) \le \omega_{x_0}(|x-x_0|).
	$$	
	where $x$ is within a left neighborhood of $x_0$.
\end{itemize}

\end{remark}
The next example demonstrates the disparity between the left and right derivative at a boundary point of the graph.

\begin{example}\label{Example_MetricDerivSVF}
	Consider $\MapR{[-1,1]}$, $F(x)=[0, 1+x]$, $x_0=0$.
	
	First, we evaluate the right metric derivative $D^M_{+} F(0)|_{y}$, for all $y\in F(0)$. 
	The set $[0,h]^M F$, ${0<h<1}$ is derived from the following collection of metric pairs 
	$$
	\Pair{F(0)}{F(h)}= 	\big \{(y,y) , \ (1, \tilde{y}):\  y\in[0,1),\, \tilde{y}\in[1, 1+h]\big \}.
	$$
	For $y\in [0,1)$ we have $[0,h]^M F|_{y} = \{0\}$ and $D^M_+ F(0)|_{y}=\lim\limits_{h\to 0}[0,h]^MF|_{y}=\{0\}$, 
	while for $y=1$ we obtain
	$$
	[0,h]^M F|_{1}=\left \{\frac{\tilde{y}-1}{h}:\,  \tilde{y} \in [1, 1+h] \right \}=[0,1]\quad , 
	\; D^M_+ F(0)|_{1}=\lim_{h\to 0}[0,h]^MF|_{1}=[0,1].
	$$
Now we calculate $\displaystyle D^M_{-} F(0)|_{y}$,  for all $y\in F(0)$. First note that for any $ y\in[0,1) $ and any $0<h<1-y$ we have 
	$ (y,y) \in  \Pair{F(-h)}{F(0)}$, then we obtain ${[0,-h]^M F|_{y} = \{0\}}$ and consequently $D^M_- F(0)|_{y}=\{0\}$.
	If $y=1$, then for any $0<h<1$ the metric pair $(1-h, 1) \in  \Pair{F(-h)}{F(0)} $ leads to 
	$$
	[0,-h]^M F|_{1}=\left \{\frac{1-h-1}{-h}=1\right \} \; {\mbox and } 
	\;\, D^M_- F(0)|_{1} = \{1\}.
	$$ 
	Thus $ D^M_{+} F(0)|_{y} =  D^M_{-} F(0)|_{y}=\{0\} $ for $ y\in[0,1) $,  while for the boundary point $(0,1)\in Graph(F)$ we have  ${D^M_{+} F(0)|_{1} \neq  D^M_{-} F(0)|_{1} }$.\\	
	\noindent Note that this multifuction is metrically differentiable at $x=0$ according to Definition~\ref{def_metr_deriv}.
\end{example}

In the following, we continue with Example~\ref{ex_ddif1} and show that $F$ from this example is metrically differentiable, with the left and right derivatives being equal.

\begin{example} \label{ex_ddif2} 
	Let $F:(0,2)\to \Kone$ be defined as in Example~\ref{ex_ddif1}, namely $F(x)=\{ x^\alpha,x^\beta \}$.
	Then, for $x_0=1$ we get
	${
		D^M_{+} F(1)|_{1} =\lim\limits_{h\to 0^+} [1,1+h]^M F|_{1} = 
		\{\alpha ,\beta\} ,
	}$ 
	while calculating for ${x_0\neq 1}$ we obtain
	$$
		D^M_{+} F(x_0)|_{x_0^\alpha}=\lim\limits_{h\to 0^+} [x_0,x_0+h]^M F|_{x_0^{\alpha}} =\{\alpha x_0^{\alpha-1}\}\quad \hbox{and}\quad D^M_{+} F(x_0)|_{x_0^\beta}=\{\beta x_0^{\beta-1}\}.
	$$  
	Thus, for each point $x\in(0,2)$, $x\neq 1$, $F(x)$ consists of two points, and 
	there is only one vector in $D^M_{+} F(x)|_{y}$ for $y \in F(x)$.
	In case $x=1$ there is only one point $y=1$ in $F(1)$, yet $D^M_{+} F(1)|_1=\{\alpha,\beta \}$.
	
	\noindent It is straightforward to verify that for any ${y\in F(x)}$ and ${x \in (0,2)}$, ${ D^M_{-} F(x)|_{y} =  D^M_{+} F(x)|_{y}}$.
\end{example} 

Now we extend the function of Example~\ref{ex_ddif1} to a function with values in~$\Ktwo$ and derive its metric left and right derivatives. 

\begin{example}\label{ex_ddif3}
	Let $F:(0,2)\to \Ktwo$ be defined by 
	$$
	F(x)=\left \{ \begin{pmatrix} x^\alpha\\ x^{\beta} \end{pmatrix} ,
	\begin{pmatrix} x^{\alpha+1}\\ x^{\beta+1} \end{pmatrix} \right \}\, \; \alpha,\beta \in \N,\, \alpha\neq \beta.
	$$
	First we evaluate the metric divided differences.
	For $x \in (0,2)$ and $|h|$ sufficiently small, the collection of all metric pairs $\Pair{F(x)}{F(x+h)}$ for $x, x+h \le 1$ or $x, x+h \ge 1$  is
		\begin{align*}
			\Pair{F(x)}{F(x+h)} = 
			\left \{ 
			\left (\begin{pmatrix} x^\alpha\\ x^{\beta^{\bf }} \end{pmatrix} , 
			\begin{pmatrix} (x+h)^\alpha\\ (x+h)^{\beta} \end{pmatrix} 
			\right )\ , \
			\left (\begin{pmatrix} x^{\alpha+1}\\ x^{\beta+1} \end{pmatrix} , 
			\begin{pmatrix} (x+h)^{\alpha+1}\\ (x+h)^{\beta+1} 	\end{pmatrix} 
			\right )
			\right \}.
		\end{align*} 
	Denote $p_1=\begin{pmatrix} x^\alpha\\ x^{\beta} \end{pmatrix} $,
	$p_2=\begin{pmatrix} x^{\alpha+1}\\ x^{\beta+1} \end{pmatrix} $.
	For $x \neq 1$ we have 
	\begin{align*}
			[x,x+h]^M F|_{p_1} = 
			\left \{ \begin{pmatrix} \frac{(x+h)^\alpha-x^\alpha}{h_{\bf }}\\ \frac{(x+h)^{\beta}-x^{\beta}}{h} \end{pmatrix} \right \}\quad ,\quad
			[x,x+h]^M F|_{p_2} = 
			\left \{ \begin{pmatrix} \frac{(x+h)^{\alpha+1^{\bf }}-x^{\alpha+1}}{h_{\bf }}\\ \frac{(x+h)^{\beta+1}-x^{\beta+1}}{h}_{} \end{pmatrix} \right \}.
	\end{align*}	
If $x=1$, the set $F(1)$ consists of a single element $p=\begin{pmatrix} 1\\1 \end{pmatrix} $, thus we obtain
$$
	[1,1+h]^M F|_{p} = 
\left \{ 
\begin{pmatrix} \frac{(1+h)^\alpha-1}{h_{\bf }}\\ \frac{(1+h)^{\beta}-1}{h} \end{pmatrix}\, , \,
\begin{pmatrix} \frac{(1+h)^{\alpha+1^{\bf }}-1}{h_{\bf }}\\ \frac{(1+h)^{\beta+1}-1}{h}_{} \end{pmatrix}.
\right \}
$$
Now we find $D^M_{+} F(x)|_{p}$ for every $p\in F(x)$ and $x\in (0,2)$. If $x \neq 1$ we get
$$
	D^M_{+} F(x)|_{p_1}= \lim_{h\to 0^+} [x,x+h]^M F|_{p_1} = 
		\left \{ \begin{pmatrix} \alpha x^{\alpha-1} \\ \beta x^{\beta-1} \end{pmatrix}	 \right \}\quad  {\mbox{and}} \quad 
	D^M_{+} F(x)|_{p_2}= 
		\left \{ \begin{pmatrix} {(\alpha+1)} x^{\alpha} \\ {(\beta+1)} x^{\beta} \end{pmatrix}	 \right \}.
$$
Calculations for $x = 1$ result in the set 
$
D^M_{+} F(x)|_{p} = 
\left \{ \begin{pmatrix} \alpha \\ \beta \end{pmatrix}\, , \, \begin{pmatrix} \alpha+1 \\ \beta+1 \end{pmatrix} 
\right \}
$

\noindent Just as in Example~\ref{ex_ddif2} here we also have $D^{M}_{-}  F(x)|_{p}= D^M_{+} F(x)|_{p}$\,  for any ${p\in F(x)}$ and ${x \in (0,2)}$.
\end{example} 
\smallskip

The following properties of the metric one-sided derivatives are derived from their definition. 
\begin{theo}\label{metderiv}
Let $F:(a,b)\to \Kn$ be metrically differentiable from the right (left) at $x_0\in (a,b)$. 
\begin{itemize}
\item[(i)] 
  $F$ is right (left) continuous at $x_0$.
\item[(ii)] 
  If $F={\rm const}$, then \, $D^M_{+} F(x)|_y=\{0\}$  $(\, D^M_{-} F(x)|_y=\{0\}\, )$  for all $y\in F(x)$, $x \in (a,b)$.
\item[(iii)]
  $D^M_{+} F(x_0)|_{y_0} \neq \{0\}$ $(\, D^M_{-} F(x_0)|_{y_0} \neq \{0\}\, )$ only for $(x_0,y_0)$ on the boundary of $Graph(F)$.
\end{itemize}
\end{theo}
\begin{proof}
We prove the claims for the right derivative. The proofs for the left derivative are analogous.
\begin{itemize}
\item[(i)] 
By Definition  \ref{def_metr_deriv}~$(i)$, 
$$
\lim_{x\to x_0^{+}} \sup_{y_0\in F(x_0)} 
\haus\left(D^M_{+} F(x_0)|_{y_0}, \left \{\frac{y-y_0}{x-x_0}\ : \ (y_0,y)\in \Pair{F(x_0)}{F(x)} \right \}\right) =0.
$$
Multiplying the two sets above by $x-x_0$ and using the second relation in \eqref{prop:haus}, 
we get 
$$
\lim_{x\to x_0^{+}} \sup_{y_0\in F(x_0)} 
\frac{1}{|x-x_0|}\haus\left((x-x_0)D^M_{+} F(x_0)|_{y_0}, \{ {y-y_0}: \ (y_0,y)\in \Pair{F(x_0)}{F(x)} \}\right) =0,
$$
and it follows that
$$
\lim_{x\to x_0^{+}} \sup_{y_0\in F(x_0)} \haus\left((x-x_0)D^M_{+} F(x_0)|_{y_0}, \{ {y-y_0}: \ (y_0,y)\in \Pair{F(x_0)}{F(x)} \}\right) =0.
$$
Since $D^M_{+} F(x_0)$ is compact (see Remark~\ref{rem_3}~(iii)), we get 
$$
\displaystyle \lim_{x\to x_0^{+}} \sup_{y_0\in F(x_0)} \|(x-x_0)D^M_{+} F(x_0)|_{y_0} \|=0,
$$
which  leads to
$$
\lim_{x\to x_0^{+}} \sup \{ |y-y_0|\ : \ (y_0,y)\in \Pair{F(x_0)}{F(x)} \}=0.
$$
Thus by \eqref{haus_MetrPair}
$$
\lim_{x\to x_0^{+}}\haus(F(x), F(x_0))=0,
$$
which proves the continuity from the right of $F$ at $x_0$. 
\item[(ii)] 
This property follows from the definition of the metric divided difference and the observation 
that if $F$ is constant, then for any two points $x_0,x_1$, 
$$\Pair{F(x_0)}{F(x_1)}=\{(y,y): y\in F(x_0) \}$$
and therefore  $[x_0,x_1]^M F|_{y}=\{0\}$ for any $y\in F(x_0)$.
\item[(iii)]
Let $(x_0,y_0)$ be an interior point of $Graph(F)$. By Remark~\ref{Remark_onDD}~(ii), $[x_0,x]^M F|_{y_0}=\{0\}$ for $x$ in a small enough neighborhood of $x_0$, implying that  $D^M_{+} F(x_0)|_{y_0} = \{0\}$. 
\end{itemize}
\end{proof}


\section{The local metric linear approximant and $\alpha$-differentiability}

In this section we define and study the rate of approximation of the local
metric linear approximant of the multifunction $F$ at a point $x_0\in (a,b)$.
\begin{definition} \label{def_MLA}
Let $x_0 \in (a,b)$ and $F$ be metrically differentiable at $x_0$. 
\begin{enumerate}
	\item [(i)] The {\bf\textsl{right local metric linear approximant} of $F$ at $x_0$ anchored 
	at $y_0 \in F(x_0)$} is defined by
	$$
	L^M_+ F|_{y_0}(x)=\{y_0\}+(x-x_0)D^M_+ F(x_0)|_{y_0}, \qquad x\ge x_0,\ x\in \R.
	$$
	The {\bf\textsl{ right local metric linear approximant of $F$ at $x_0$}} is the set-valued function
	$$
	L^M_+ F(x)=\bigcup\limits_{y\in F(x_0)} L^M_+ F|_{y}(x).
	$$
	\item [(ii)]  Similarly, we define the {\bf\textsl{left local metric linear approximant anchored 
	at $y_0 \in F(x_0)$}, } $L^M_{-} F|_{y_0}(x)$, and the {\bf\textsl{left local metric linear approximant  }} $ L^M_- F(x)$.
	\item [(iii)] The {\bf\textsl{(two-sided) local metric linear approximant} of $F$ at $x_0$ anchored at ${y_0 \in F(x_0)}$} is
$$
L^M F|_{y_0}(x)= \begin{cases}
	L^M_+ F|_{y_0}(x), & x \ge x_0, \\
	L^M_- F|_{y_0}(x), & x < x_0.
\end{cases}
$$
The {\bf\textsl{(two-sided) local metric linear approximant }} of $F$ at $x_0$ is the set-valued function
$$
L^M F(x)= \begin{cases}
	L^M_+ F(x), & x \ge x_0, \\
	L^M_- F(x), & x < x_0.
\end{cases}
$$
\end{enumerate}
\end{definition}
\noindent Obviously $L^M_+ F(x_0)=L^M_- F(x_0)=L^M F(x_0)=F(x_0)$.
\medskip

The following observations follow from this definition.
\begin{theo}\label{prop:metrLA}
Let $F:\R\to \Kn$ be metrically differentiable at~$x_0$. Then for~$x\in (a,b)$ 
\begin{itemize}
\item[(i)] \ $\haus\left(L^M F|_{y_0}(x),\{y_0\}+(x-x_0)[x_0,x]^M F|_{y_0}\right) = o(|x-x_0|)$  uniformly in ${y_0\in F(x_0)}$.
\item[(ii)] \ $\haus(F(x),L^M F(x)) = o(|x-x_0|)$.
\end{itemize}
\end{theo}
\begin{proof}${}$\\
 (i) follows directly from Remark~\ref{rem_3}~(iv) and~\eqref{prop:haus}. \\
 (ii) follows from (i), \eqref{haus_pr_uni} and Remark~\ref{Remark_onDD} (iii).
\end{proof}

Note that (ii) in the above proposition implies that $L^M_+ F(x)$ is a right directive 
in the sense of Artstein~\cite[Definition~5.5]{ART:95} or local approximation of first order in the sense of Nikolskiĭ~\cite[Definition~1]{NIK:93}.

To obtain error estimates of order $1+\alpha$ ($\alpha>0$) for the local linear approximant,
we introduce the notion of $\alpha$-differentiability of a set-valued function at a point.
\begin{definition}\label{def:unifMD}
	${}$
	\begin{enumerate}
	\item [{(i)}] Let a multifunction $F$ be metrically differentiable from the right at $x_0$.
		We call $F$ {\bf \textsl{right metrically $\alpha$-differentiable}} at $x_0$ with $\alpha>0$,
		if there is a constant $L>0$ and a right neighborhood $U$ of $x_0$ such that for all $x\in U$ 
		$$
		\sup_{y \in F(x_0)} \haus([x_0,x]^M F|_y,D^M_+ F(x_0)|_y) \le L |x-x_0|^\alpha.
		$$
		$F$ is called {\bf\textsl{strongly right metrically differentiable}} at $x_0$ if $\alpha=1$, namely
		$$
		\sup_{y \in F(x_0)} \haus([x_0,x]^M F|_y,D^M_+ F(x_0)|_y) \le L|x-x_0|.
		$$
	\item [{(ii)}] Similarly we define these notions from the left.
	\item [{(iii)}] If $F$ is right and left metrically $\alpha$-differentiable at $x_0$, $\alpha >0$,
	then we say that $F$ is {\bf \textsl{metrically $\alpha$-differentiable}} at $x_0$ with $\alpha >0$.
	We say that  $F$	is {\bf\textsl{strongly metrically differentiable}} at $x_0$ if it is strongly metrically differentiable from the right and from the left at $x_0$.
\end{enumerate}
\end{definition}

\begin{remark}
It is easy to verify that any real-valued function with H\"older continuous derivative with exponent $\alpha \in (0,1]$ 
in a neighborhood of $x_0$ is metrically $\alpha$-differentiable at~${x_0}$. 
\end{remark}

The next theorem shows that the local linear approximant of a metrically $\alpha$-differentiable multifunction
has order of approximation $1+\alpha$. 

\begin{theo}\label{th:order_LM}
If $F$ is metrically $\alpha$-differentiable at $x_0$, then for $x$ from a neighborhood of $x_0$
$$
\haus \left ( F(x),L^M F(x) \right ) \le L|x-x_0|^{1+\alpha}.
$$
\end{theo} 
\begin{proof}
We consider the case $x>x_0$. For $x<x_0$ the proof is similar.  

By Definition \ref{def_MLA}, by Remark~\ref{Remark_onDD} (iii) and by \eqref{haus_pr_uni} we get for any $y\in F(x_0)$
\begin{align*}
\haus(F(x),L^M F(x)) & = \haus(F(x),L^M_+ F(x)) \\& 
\le \sup_{y\in F(x_0)} \haus \left ( \{y\}+(x-x_0)[x_0,x]^M F|_y\ , L^M_+ F|_{y}(x) \right ).
\end{align*}
Then, using  properties \eqref{prop:haus} of the Hausdorff distance and Definitions~\ref{def_MLA}, \ref{def:unifMD}, 
we get
$$
\\haus(F(x),L^M_+ F(x))\le |x-x_0|\sup_{y\in F(x_0)} \haus([x_0,x]^M F|_y, D^M_+ F(x_0)|_y) 
\le L |x-x_0|^{1+\alpha},
$$
which proves the claim for $x>x_0$ . 
\end{proof}

A direct consequence of Theorem~\ref{th:order_LM} is
\begin{cor}\label{corol_1}
If $F$ is strongly metrically differentiable at $x_0$, then 
$$\haus(F(x),L^M F(x)) = O(|x-x_0|^2).$$
\end{cor}

In the next example we present a strongly metrically differentiable multifunction at a point $x=0$ and construct its local metric linear approximant.
\begin{example}\label{example_StronglyMetrDif}
	Let $F:[-1,1] \to \Kone$  be defined as 
	$$
	F(x)= \begin{cases}
		[0, 2-x^2]\,  , & x \in [-1,0], \\
		[0,2+ x] \cup \{-x^2\}\,  , & x \in [0,1].
	\end{cases}
	$$
	First we show that $F$ is strongly right metrically differentiable at $x=0$. We start by calculating the metric divided differences $[0,h]^M F|_{y}$, $y\in F(0)=[0,2]$, $0<h\le 1$
	
	For $y\in (0,2)$, by Remark~\ref{Remark_onDD} (ii) and by Definition~\ref{def_metr_deriv} we have 
	$$
	[0,h]^M F|_{y} = \{0\}\quad \hbox{and} \quad  D^M_+ F(0)|_{y}=\{0\}.
	$$
	If $y=2$, similarly to the calculations in Example~\ref{Example_MetricDerivSVF} we obtain 
	$$
	[0,h]^M F|_{2} =[0,1] \quad \hbox{and} \quad D^M_+ F(0)|_{2}=[0,1].
	$$
	Thus for ${y\in (0,2]}$ we have ${\haus \left ( [0,h]^M F|_{y}, D^M_+ F(0)|_{y} \right )=0}$.
	
	When $y=0$, ${0<h\le 1}$, one can easily see that the only metric pairs in $\Pair{F(0)}{F(h)}$ containing $y=0$ are $(0,0)$ and $(0, -h^2)$, thus obtaining 
	\begin{equation}\label{formula_1}
		[0,h]^M F|_{0} = \left \{0, \frac{-h^2-0}{h} \right \} \quad \mbox{and} \quad D^M_+ F(0)|_{0} = \lim\limits_{h\to 0^+} [0,h]^M F|_{0} = \{0\}.
	\end{equation}
	This leads to 
	$\haus( [0,h]^M F|_{0}, D^M_+ F(0)|_{0})=h$.
	
	Note that the elements in $[0,h]^M F|_{0}$ are the values of the first divided differences of the real-valued functions (boundary functions) $s_1(x)=0$, $s_2(x)=-x^2$ based on $x_0=0$, $x_1=h$. 
	Accordingly, the single element of $D^M_+ F(0)|_{0}$ is the value of right derivatives of $s_1(x), s_2(x)$ at $x=0$. 
	
	Summing up the above, we conclude that 
	$$
	\sup_{y \in F(0)} \haus([0,h]^M F|_y,D^M_+ F(0)|_y) \le h,
	$$
	therefore $F$ is strongly right metrically differentiable at $x=0$.
	\medskip 
	
	Now we check from the left of $x=0$. For $ y\in[0,2) $, we again refer to the calculations in Example~\ref{Example_MetricDerivSVF} from  which one can easily get that 
	${[-h,0]^M F|_{y} = \{0\}}$ for any $0<h<\sqrt{2-y}$.
	Consequently $D^M_- F(0)|_{y}=\{0\}$ and $\haus([-h,0]^M F|_y,D^M_- F(0)|_y)=0$. 
	
	When $y=2$, computations similar to those used in~\eqref{formula_1}, 
	lead to ${[-h,0]^M F|_{2} = \{h\}}$ and ${D^M_- F(0)|_{2}=\{0\}}$. Thus we can derive
	$$
	\sup_{y \in F(0)} \haus([-h,0]^M F|_y,D^M_- F(0)|_y) \le h,
	$$
	and as a result, we conclude that $F$ is strongly left metrically differentiable at $x=0$. Thus $F$ is strongly metrically differentiable at $x=0$.
	
	Finally we construct the (two-sided) local metric linear approximant of $F$ at $x=0$ in accordance with Definition~\ref{def_MLA}.
	Consider first $x\ge 0$. Since $D^M_+ F(0)|_{y}=\{0\}$ for any $y\in [0,2)$, we get 
	$${L^M_+ F|_{y}(x) = \{y\}+(x-0)D^M_+ F(0)|_{y} = \{y\}}.$$
	As for $y=2$, we have $D^M_+ F(0)|_{2}=[0,1]$ and then ${L^M_+ F|_{2}(x) = \{2\} +x\cdot [0,1] = [2,2+x] }$. Thus
	$$
	L^M_+ F(x)=\bigcup\limits_{y\in [0,2]} L^M_+ F|_{y}(x)  = [0,2) \bigcup \ [2,2+x] = [0,2+x].
	$$
	If $x < 0$, then $L^M_- F(x)=[0,2]$, since $D^M_- F(0)|_{y}=\{0\}$, $\forall y \in [0,2]$. Summing up the results we obtain
	$$
	L^M F(x) = \begin{cases}
		[0, 2]\,  , & x < 0, \\
		[0,2+ x]\,  , & x \ge 0.
	\end{cases}
	$$
	It is easy to see that for $|h|\le 1$
	$$
	\haus\left (F(h),L^M F(h) \right) = h^2
	$$
	in accordance with Corollary~\ref{corol_1}.
\end{example}

{\bf Acknowledgment}

The authors thank Ron Goldman for initial discussions on possible definitions of derivatives of SVFs and for suggesting the term "anchored".



\end{document}